\documentclass[a4paper,12pt]{article}
\usepackage[english]{babel}
\usepackage[utf8]{inputenc}
\usepackage[T1]{fontenc}
\usepackage{textcomp} 
\usepackage{amsmath}
\usepackage{amsfonts}
\usepackage{amsthm}
\usepackage{mathrsfs}
\usepackage{amssymb}
\usepackage{csquotes}
\usepackage{graphicx}
\usepackage[pagewise]{lineno}\nolinenumbers

\theoremstyle{definition}
\newtheorem{Def}{Definition}[section]

\theoremstyle{plain}
\newtheorem{Pro}[Def]{Proposition}

\newtheorem{Lem}[Def]{Lemma}

\newcommand{\R}{\mathbb{R}}
\newcommand{\Q}{\mathbb{Q}}
\newcommand{\Z}{\mathbb{Z}}
\newcommand{\N}{\mathbb{N}}

\title{The Thomae Function: Fractal Insights}
\author{T. Lamby\footnote{University of Luxembourg, Department of Mathematics, Maison du Nombre, 6, avenue de la Fonte, L-4364 Esch-sur-Alzette. Orcid ID: 0000-0001-6549-0678. thomas.lamby@uni.lu}, Samuel Nicolay\footnote{Universit\'e de Li\`ege, D\'epartement de math\'ematique -- zone Polytech 1, 12 all\'ee de la D\'ecouverte, B\^at. B37, B-4000 Li\`ege. Orcid ID: 0000-0003-0549-0566. S.Nicolay@uliege.be}}

\begin{document}

\maketitle

\begin{abstract}
This article examines the Thomae function, a paradigmatic example of a function that is continuous on the irrationals and discontinuous elsewhere. Defined for a parameter $\theta>0$, it exhibits a rich self-similar structure and intriguing regularity properties. After revisiting its fundamental characteristics, we analyze its H\"older continuity, emphasizing the interplay between its discrete spikes and its behavior on dense subsets of the real line. This study provides a refined perspective on the irregularity of the Thomae function, using classical analytical tools to elucidate its fractal nature.
\end{abstract}

\section{Introduction}
The Thomae function - also known as the popcorn function among other names - has long served as a striking example in real analysis, illustrating the delicate interplay between continuity and discontinuity. Explicitly introduced by Thomae in 1875 \cite{Thomae} as a pedagogical example in the context of the formalisation of continuity, this function is continuous at every irrational point and discontinuous at every rational one. Since it is continuous almost everywhere (with respect to Lebesgue measure), the function is Riemann integrable, with vanishing integral. It also provides a neat illustration of Blumberg’s theorem, which asserts that for any function $f:\R\to \R$, there exists a dense subset of $\R$ on which $f$ is continuous \cite{Blumberg}.

Let us now introduce a slight generalisation of this function, preserving its essential properties. Denote by by $p\wedge q$ the greatest common divisor of two integers $p$ and $q$; thus $p\wedge q=1$ indicates that $p$ and $q$ are coprime. Unless otherwise specified, any rational number $x$ written as $x=p/q$ with $p\in \Z$ and $q\in \N$ will be assumed to satisfy $p\wedge q=1$. The Thomae function can then be defined as
\[
 f_\theta (x) =
 \left\{\begin{tabular}{ll}
  $1$ & if $x=0$ \\
  $q^{-\theta}$ & if $x$ is rational with $x=p/q$ \\
  $0$ & if $x$ is irrational
 \end{tabular}\right.,
\]
with $\theta=1$. The limiting case $\theta=0$ corresponds to the Dirichlet function, whereas for $\theta < 0$, the function fails to be locally bounded anywhere. For $\theta>0$, it exhibits the remarkable property of being continuous on the irrational numbers while discontinuous at every rational point. This duality, together with its quasi self-similar structure, makes the Thomae function a paradigmatic object in the study of irregular behavior in real analysis. In what follows, we focus on the case $\theta>0$; more general variants are discussed in \cite{Beanland}.
\begin{figure}
\begin{center}
\includegraphics[scale=.25]{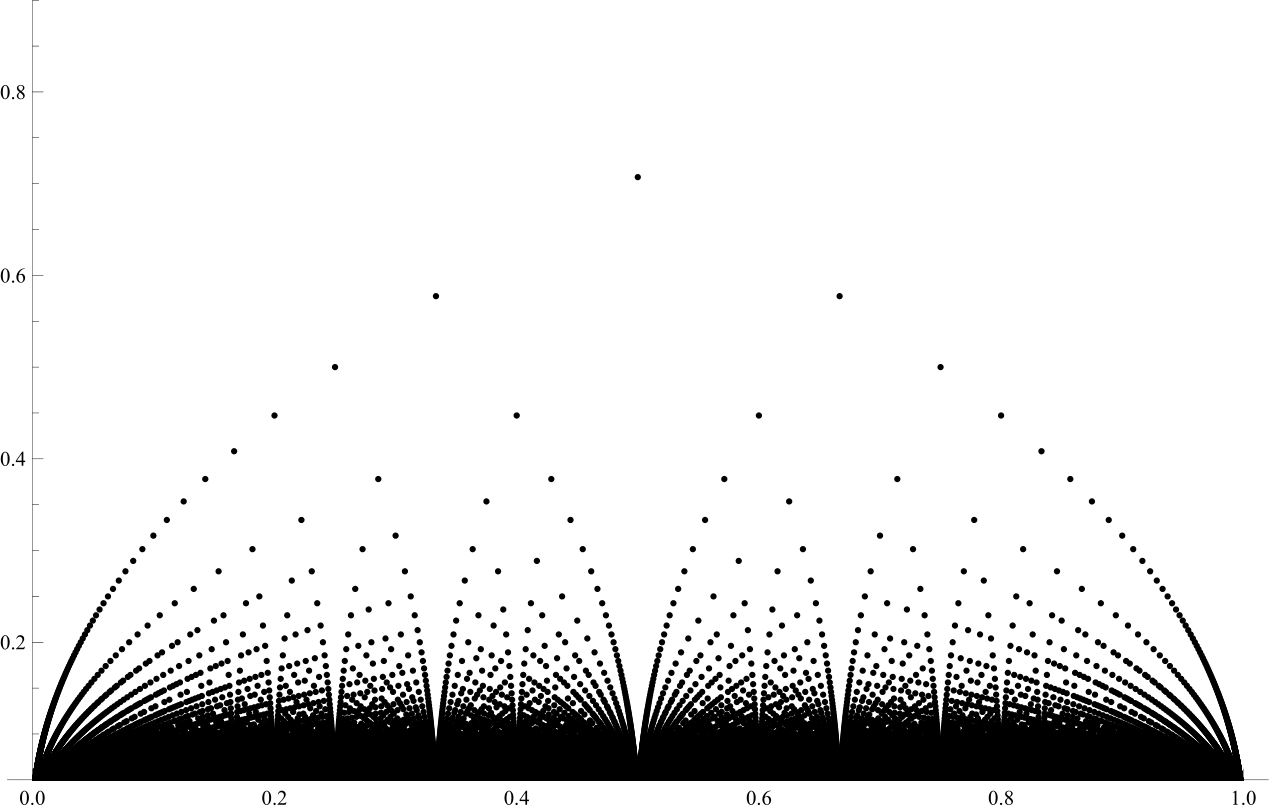}
\includegraphics[scale=.25]{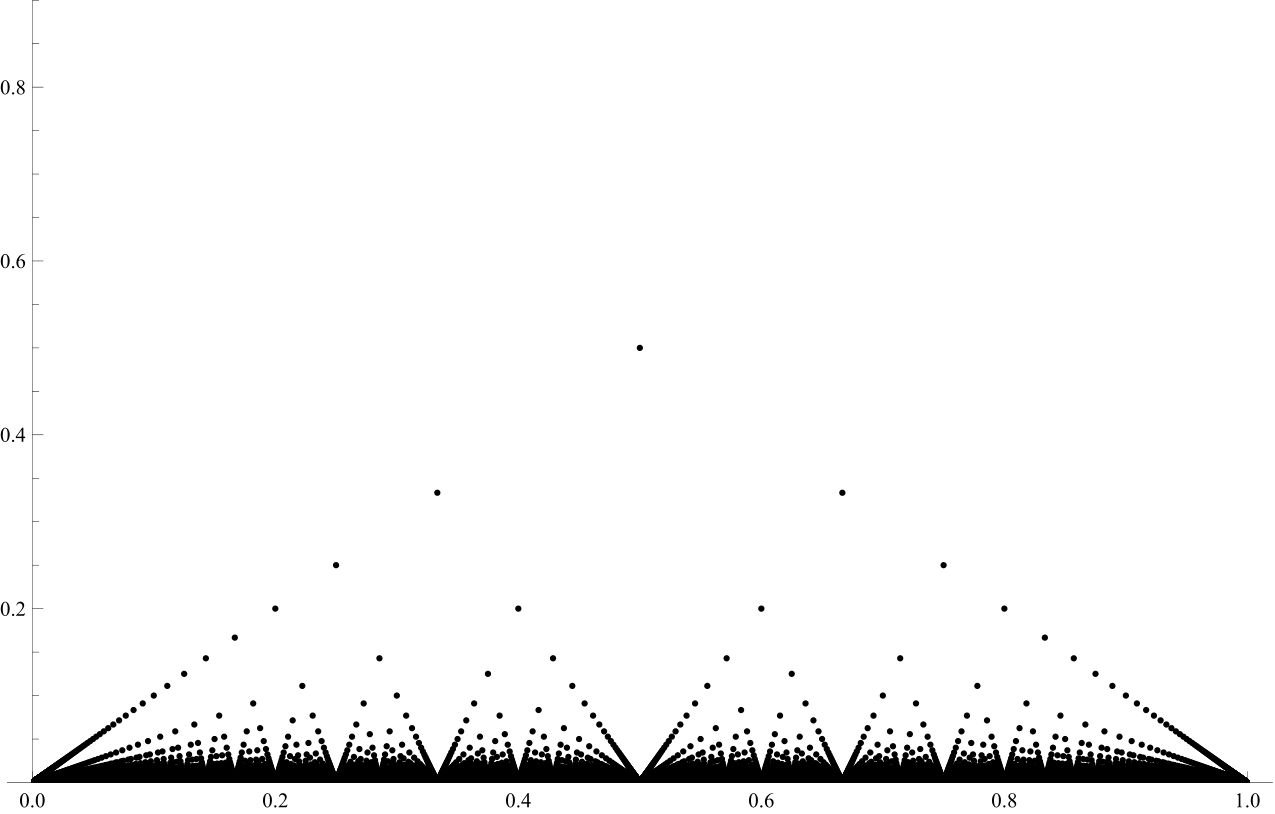}
\includegraphics[scale=.25]{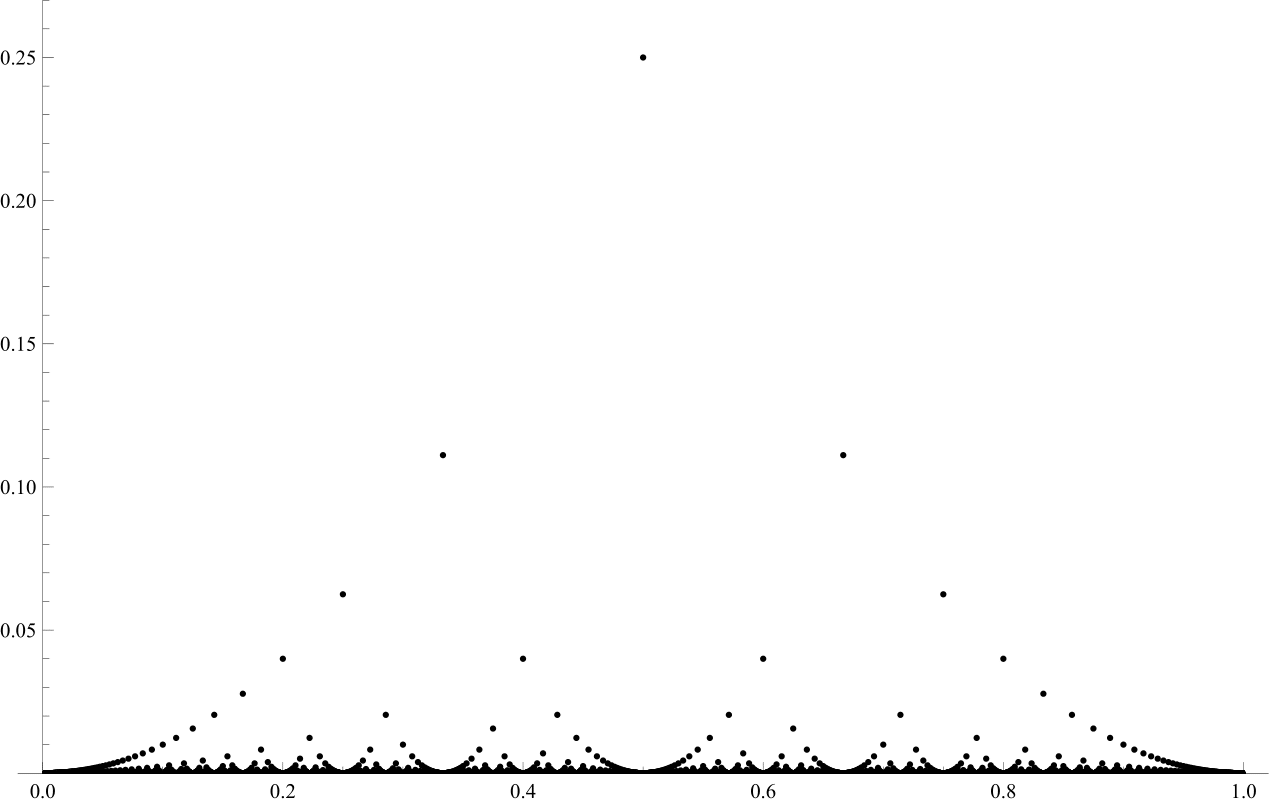}
\end{center}
\caption{Representation of the function $f_\theta$ on $(0,1)$ for $\theta=1/2$, $1$ and $2$.}
\end{figure}

Beyond its classical role in real analysis, the Thomae function has found relevance in broader mathematical and applied contexts. Recent studies have highlighted analogies between its spiked structure and distributions observed in empirical datasets, particularly in biology and clinical research \cite{Trifonov}.

This article focuses on the H\"older regularity of the Thomae function, a key aspect of its behavior. The first section reviews its fundamental properties, offering a detailed account of its defining characteristics and self-similar nature. The subsequent discussion delves into the exponent of irrationality, providing further insight into the function's intricate behavior. In the fourth section, we analyze the function's H\"older regularity, uncovering insights into its fractal-like properties through contemporary mathematical tools. By bridging its classical foundations with these contemporary perspectives, we aim to highlight both the theoretical elegance and the deeper structural nuances of this remarkable function.

\section{Fundamental properties}
We begin by recalling some well-known properties of the Thomae function.

The Thomae function is periodic with period $1$.
\begin{Pro}
For any $x\in \R$, $f_\theta(x+1)=f_\theta(x)$.
\end{Pro}
\begin{proof}
If $x$ is irrational, so is $x+1$ and therefore $f_\theta(x+1)=f_\theta(x)=0$.

If $x=p/q$, then $x+1=\frac{p+q}{q}$. Since a number divides $p$ and $q$ if and only if it divides $p+q$ and $q$, the condition $p\wedge q=1$ implies $(p+q)\wedge q=1$. Thus, for such a rational $x$, $f_\theta(x+1)=f_\theta(x)=q^{-\theta}$.
\end{proof}

Regarding continuity, the following result holds.
\begin{Pro}
The function $f_\theta$ is discontinuous at rational points and continuous at irrational points.
\end{Pro}
\begin{proof}
If $x$ is rational, let $s$ be an irrational number, and define $x_j= x+ s/j$ for $j\in \N$. Clearly $x_j\to x$, but since $x_j$ is irrational for all $j$, $f_\theta(x_j) \not\to f_\theta(x)$.

If $x$ is irrational, assume $x\in (0,1)$. Given $\varepsilon>0$, choose $n\in \N$ be such that $n^{-\theta}<\varepsilon$. For $j\in \{1,\dotsc,n\}$, define $m_j=\sup \{m\in \N_0: m < jx\}$, and set
\[
 \delta_j = \inf \{ |x- \frac{m_j}{j}|, |x- \frac{m_j+1}{j}| \}.
\]
Let $\delta= \inf_{1\le j\le n} \delta_j$. If $y=p/q$ is rational and $y\in (x-\delta, x+\delta)$, then $q>n$, so $f_\theta(y)<n^{-\theta} <\varepsilon$. If $y$ is irrational, $f_\theta(y)=0<\varepsilon$. Thus, $|x-y|< \delta$ implies $|f_\theta(x)-f_\theta(y)|< \varepsilon$, proving that $f_\theta$ is continuous at $x$.
\end{proof}

From the preceding result, $f_\theta$ is not differentiable at rational points. Let us now establish that for $\theta\in (0,2]$ the function is not differentiable at irrational points. A stronger result will be demonstrated in the next section.
\begin{Pro}
For $\theta\in (0,2]$, the function $f$ is not differentiable at any point.
\end{Pro}
\begin{proof}
Let $x$ be an irrational number. By Hurwitz's theorem \cite{Hurwitz,Hardy:08,Bugeaud}, there exists a sequence $(x_j)_{j\in \N}$ of rational numbers converging to $x$, such that $x_j=p_j/j$ with $p_j\wedge j=1$ and $|x- x_j| < \frac{1}{\sqrt{5} j^2}$. Then
\[
 | \frac{f_\theta(x) - f_\theta(x_j)}{x- x_j}|
 > \frac{j^{-\theta}}{1/(\sqrt{5} j^2)}
 = \sqrt{5} j^{2-\theta}.
\]
This ensures that $f_\theta'(x)$ cannot be equal to zero. However, by irrational approximation, if $f_\theta'(x)$ exists, it must be zero.
\end{proof}
The differentiability in the case $\theta>2$ is considered in \cite{Beanland}, where a more general scenario is studied. Whether $f_\theta$ is differentiable at $x$ or not involves conditions on the irrationality exponent of $x$ (see Definition~\ref{def:irex}).

\section{Irrationality exponent}
To study the regularity of the Thomae function, we must first understand how well an irrational number can be approximated by rational numbers.
\begin{Def}\label{def:irex}
The irrationality exponent $\tau(x)$ of an irrational number $x$ is the supremum of the real numbers $\tau$ for which the inequality
\[
 |x - \frac{p}{q}| < \frac{1}{q^\tau} 
\]
has infinitely many solutions in non zero integers $p$ and $q$.
\end{Def}
For any irrational number $x$, we have $\tau(x)\ge 2$ and the irrationality exponent of almost every number $x$ (with respect to the Lebesgue measure) is equal to $2$ \cite{Hardy:08,Bugeaud}. Roth's theorem asserts that the irrationality exponent of any irrational algebraic number is exactly $2$ \cite{Roth,Hardy:08}. For example, one has $\tau(e)=2$ and $\tau(\pi)< 8.0161$ \cite{Hata}.

To examine the irregularity of the Thomae function, we will require the ``best'' rational approximation of an irrational number as discussed in \cite{Hardy:08}, for example. This can be achieved with the continued fraction expansion.
\begin{Def}
Let $(a_j)_{j\in \N}$ be a sequence of natural numbers. For any $n\in \N$, we define the finite continued fraction
\[
 [a_1,a_2,\dotsc, a_n] = \frac{1}{a_1+\frac{1}{a_2+ \frac{1}{\dotsb + \frac{1}{a_n}}}}
\]
and the infinite continued fraction $[a_1,a_2,\dotsc]= \lim_{n\to \infty} [a_1,a_2,\dotsc,a_n]$, provided that the previous limit exists.
\end{Def}
It can be shown that $[a_1,a_2,\dotsc]$ defines an irrational number of $(0,1)$, and that every irrational of $(0,1)$ can be uniquely represented as an infinite continued fraction \cite{Hardy:08,Bugeaud}.
\begin{Def}
The convergents of an irrational number $x\in (0,1)$ are the numbers $p_j/q_j= [a_1,a_2,\dotsc, a_j]$.
\end{Def}
The sequence of convergents $(p_j/q_j)_{j\in \N}$ converges to $x$. Moreover, if $\tau_j(x)$ denotes the number such that
\[
 |x- \frac{p_j}{q_j}|= \frac{1}{q_j^{\tau_j(x)}},
\]
we naturally have $\tau(x)= \limsup_j \tau_j(x)$.

\section{H\"older regularity}
As we aim to determine the H\"older exponent of the Thomae function at every point, we must first introduce this concept.
\begin{Def}\label{Def Holder spaces}
A locally bounded function $f$ defined on $\R$ belongs to the pointwise H\"older space $\Lambda^\alpha(x)$, with $x\in \R$ and $\alpha \geq 0$ if there exists a constant $C>0$ and a polynomial $P$ of degree less than $\alpha$ such that
\[
| f(x+h)-P(h)| \le C |h|^\alpha,
\]
for any $h$ in a neighborhood of the origin.
\end{Def}
It is easy to check that $\alpha<\alpha'$ implies $\Lambda^{\alpha'}(x) \subset \Lambda^{\alpha}(x)$.
\begin{Def}
The H\"older exponent of a locally bounded function $f$ defined on $\R$ at $x\in \R$ is defined by $H_f(x) = \sup\{ \alpha: f\in \Lambda^\alpha(x)\}$ 
\end{Def}
In other words, the H\"older exponent $H_f(x)$ is such that
\[
 f\in \bigcap_{0\le \alpha<H_f(x)}\Lambda^\alpha(x)
 \qquad\text{and}\qquad
 f\not\in \bigcup_{H_f(x)<\alpha\le 1}\Lambda^\alpha(x).
\]
The H\"older exponent offers deeper insight into the regularity of the function $f$ at $x$. If $H_f(x)\in (0,1)$, the function $f$ is continuous but not differentiable at $x$. Conversely, if $H_f(x)>1$, the function $f$ is differentiable at $x$.

Remark that, by irrational approximation, the polynomial involved in Definition \ref{Def Holder spaces} must be zero:
\begin{Lem}
Let $\theta, \alpha > 0$ and $x \in (0,1) \setminus \Q$. If $f_\theta \in \Lambda^\alpha(x)$, then the polynomial $P$ of degree less than $\alpha$ 
appearing in Definition~\ref{Def Holder spaces} must necessarily be the zero polynomial.\end{Lem}
\begin{proof}
By hypothesis, there exists a constant $C>0$ such that
\[
 |f_\theta(x+h)-P(h)| \le C |h|^\alpha,
\]
for all $h$ in a neighborhood of the origin. Suppose that $P(h)=\sum_{k=0}^m a_k h^k$ with $a_m\ne 0$ and $m<\alpha$. For all $n\in\N$, one has
\begin{align*}
 \frac{|f_\theta(x+1/n)-P(1/n)|}{|1/n|^\alpha}=n^\alpha |P(1/n)|.
\end{align*}
Since the right-hand side must remain bounded as $n\to \infty$, it follows that $a_0=0$. Next,
\begin{align*}
 \frac{|f_\theta(x+1/n)-P(1/n)|}{|1/n|^\alpha}
 =n^{\alpha-1} |n P(1/n)|,
\end{align*}
which implies $a_1=0$. Proceeding inductively in this manner, we find that $P(h)=a_m h^m$ and
\begin{align*}
\frac{|f_\theta(x+1/n)-P(1/n)|}{|1/n|^\alpha}=n^{\alpha-m} a_m,
\end{align*}
so that $P$ must in fact be the zero polynomial.
\end{proof}

\begin{Pro}
For $\theta>0$, the H\"older exponents of the Thomae function $f_\theta$ are given by
\[
 H_f(x) = \left\{\begin{tabular}{ll}
  $0$ & if $x$ is rational \\
  $\theta/\tau(x)$ & if $x$ is irrational
 \end{tabular}\right..
\]
\end{Pro}
\begin{proof}
Since $f_\theta$ is discontinuous at rational points, we may assume that $x \in (0,1)$ is irrational.

If $y \in (0,1)$ is also irrational, then naturally $f_\theta(x) - f_\theta(y) = 0$. Given $\varepsilon > 0$, there exists $N(\varepsilon) \in \N$ such that, for all $n \ge N(\varepsilon)$ and all $m$ coprime with $n$,
\[
 \frac{1}{|x-m/n|} \le n^{\tau(x)+\varepsilon}.
\]
Note that the function $\varepsilon\mapsto N(\varepsilon)$ is non-increasing. Hence, for all $n \ge N(1)$ and all $\varepsilon \in (0,1)$,
\[
 \frac{| f_\theta(x) -f_\theta(m/n) |}{|x-m/n|^{\theta/\tau(x)}}
 = \frac{n^{-\theta}}{|x-m/n|^{\theta/\tau(x)}}
 \le \frac{n^{\theta(\tau(x)+\varepsilon )/\tau(x)}}{n^\theta}
 \le n^{\varepsilon \theta/\tau(x)}.
\]
It follows that, for all $n \geq N(1)$,
\[
 \frac{| f_\theta(x) -f_\theta(m/n) |}{|x-m/n|^{\theta/\tau(x)}} \le 1,
\]
so that $f_\theta\in \Lambda^{\theta/\tau(x)}(x)$.

Now, let $\varepsilon > 0$ and choose $\eta\in \big(0,\frac{\tau(x)\varepsilon}{\frac{\theta}{\tau(x)}+\varepsilon}\big)$. If $p_j/q_j$ are the convergents of $x$, then, for sufficiently large $j$, we have
\[
 \frac{| f_\theta(x) -f_\theta(p_j/q_j) |}{|x-p_j/q_j|^{\frac{\theta}{\tau(x)}+\varepsilon}}
 = \frac{q_j^{-\theta}}{|x-p_j/q_j|^{\frac{\theta}{\tau(x)}+\varepsilon}}
 \ge \frac{q_j^{(\tau(x)-\eta)(\frac{\theta}{\tau(x)}+\varepsilon)}}{q_j^\theta}
 = q_j^{\beta_\varepsilon},
\]
for some $\beta_\varepsilon>0$. As $q_j\to \infty$, we deduce that $f_\theta \not\in \Lambda^{\frac{\theta}{\tau(x)}+\varepsilon}(x)$.
\end{proof}

\section{Remarks and final observations}
Let us make a few observations. The equality $\alpha= H_f(x)$ does not necessarily imply that $f\in \Lambda^\alpha (x)$. In fact, from the prevalent point of view, most functions $f$ do not belong to $\Lambda^{H_f(x)}(x)$, as a logarithmic correction is required \cite{Loosveldt:22}. However, the previous proof demonstrates that the Thomae function does belong to such a space.

For $\theta < 2$, we recover the classical result that $f_\theta$ is nowhere differentiable. The function $f_2$ is likewise nowhere differentiable, yet its H\"older exponent equals $1$ almost everywhere. For $\theta > 2$, $f_\theta$ is differentiable at a point $x$ whenever $\tau(x) < \theta$. For instance, $f_9$ is differentiable at irrational algebraic numbers, as well as at $e$, $\pi$, $\pi^2$, and $\ln(2)$.

The H\"older behavior of $f_\theta$ closely resembles that of the Brjuno function \cite{Martin}, which is unsurprising given that both functions are defined via rational approximations, see also \cite{Jaffard97}.

The H\"older spectrum of a function $f$ is defined as the map
\[
 h \mapsto \dim \{x : H_f(x) = h\},
\]
where $\dim$ denotes the Hausdorff dimension (with the convention that $\dim(\emptyset)$ is $-\infty$). Since Jarn\'ik's theorem asserts that $\dim \{x : \tau(x) = t\} = 2/t$ (see, e.g., \cite{Jaffard}), the H\"older spectrum of the Thomae function $f_\theta$ is therefore given by
\[
 h\mapsto \left\{\begin{tabular}{ll}
 $2h/\theta$ & if $h\in [0,\theta/2]$ \\
 $-\infty$ & else
 \end{tabular}\right..
\]

For $\theta < 0$, the function $f_\theta$ is nowhere locally bounded. However, unlike the Brjuno function, the so-called $p$-exponents fail entirely when applied to $f_\theta$, since its integral vanishes \cite{Leonarduzzi,Martin,Loosveldt21}.

The function $f_\theta$ studied here can be generalized using the concept of Boyd functions \cite{Loosveldt21}. A function $\phi: (0,1]\to (0,\infty)$ is called a Boyd function if it is continuous and
\[
 0<\underline{\phi}(x):=\inf_{s\leq 1}\frac{\phi(xs)}{\phi(s)}
 \leq\overline{\phi}(x):=\sup_{s\leq 1}\frac{\phi(xs)}{\phi(s)}<\infty
\]for any $t< 1$. The \emph{lower} and \emph{upper indices} of $\phi$ are defined by
\[
\underline{s}(\phi) = \lim_{x \to 0} \frac{\log \underline{\phi}(x)}{\log x} \quad \text{and} \quad \overline{s}(\phi) = \lim_{x \to 0} \frac{\log \overline{\phi}(x)}{\log x}.
\]
respectively. From the definition, it is straightforward to verify that the Boyd indices satisfy $\underline{s}(\phi)\le \overline{s}(\phi)$. Furthermore, for any $\varepsilon>0$, there exists a constant $C>0$ such that
\begin{equation*}
C^{-1} x^{\overline{s}(\phi)+\varepsilon} \le \phi(x) \le C x^{\underline{s}(\phi)-\varepsilon},
\end{equation*}
for all $t\leq 1$. These relationships enable us to adapt the preceding results to the function
\[
 f_\phi(x) = \left\{\begin{tabular}{ll}
  $1$ & if $x=0$ \\
  $\phi(1/q)$ & if $x=p/q$ \\
  $0$ & if $x$ is irrational
 \end{tabular}\right.,
\]
where $\underline{s}(\phi)= \overline{s}(\phi)=\theta>0$. For example, we can consider the Boyd function $\phi(x)=x^\theta (|\ln x|+1)^\gamma$ for some $\gamma>0$. If the indices $\underline{s}(\phi)$ and $\overline{s}(\phi)$ differ, the regularity can only be bracketed between upper and lower estimates.
\begin{figure}[h!]
\begin{center}
\includegraphics[scale=.3]{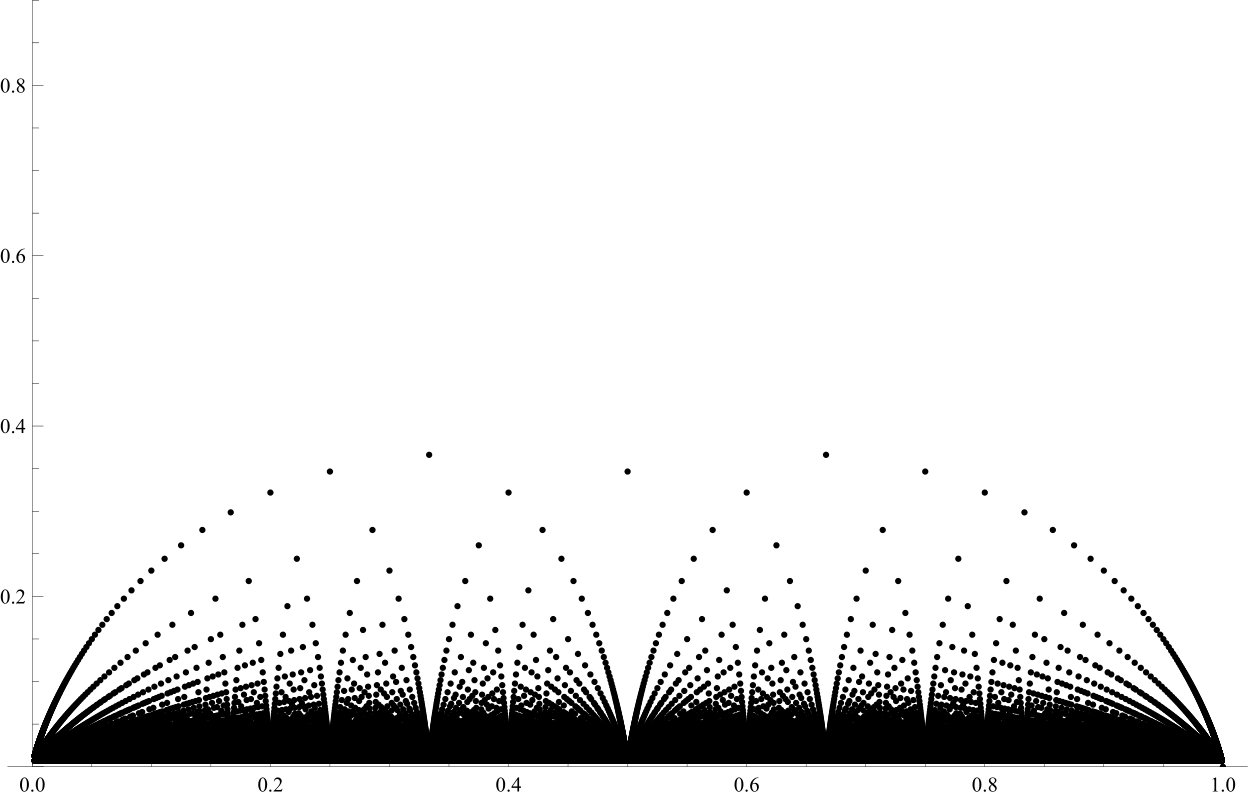}
\end{center}
\caption{Representation of the function $f_\phi$ on $(0,1)$ with $\phi(x)=t\ln(1/x)$.}
\end{figure}

In summary, the Thomae family offers a remarkably explicit example of a multifractal hierarchy governed by arithmetic properties of the real line. Its regularity, entirely determined by the irrationality exponent, illustrates the subtle interplay between number-theoretic structure and analytic smoothness. While the usual Hölder framework captures this behavior precisely, the $p$-exponents appear less adapted to such highly singular constructions, as their definition becomes degenerate in this context.

\bibliography{thomae}{}

\begin{thebibliography}{10}

\bibitem{Beanland}
K.~Beanland, J.W. Roberts, and C.~Stevenson.
\newblock Modifications of {T}homae's function and differentiability.
\newblock {\em Am. Math. Mon.}, 116(6):531--535, June 2009.

\bibitem{Blumberg}
H.~Blumberg.
\newblock New properties of all real functions.
\newblock {\em Proceedings of the National Academy of Sciences},
  8(10):283--288, October 1922.

\bibitem{Bugeaud}
Y.~Bugeaud.
\newblock {\em Distribution modulo one and diophantine approximation}.
\newblock Number 193 in Cambridge tracts in mathematics. Cambridge University
  Press, Cambridge, 2012.

\bibitem{Hardy:08}
G.H. Hardy and E.M. Wright.
\newblock {\em An introduction to the theory of numbers}.
\newblock Oxford mathematics. Oxford University Press, Oxford, sixth edition
  edition, 2008.

\bibitem{Hata}
M.~Hata.
\newblock Rational approximations to {$\pi$} and some other numbers.
\newblock {\em Acta Arith}, 63(4):335--349, 1993.

\bibitem{Hurwitz}
A.~Hurwitz.
\newblock Ueber die angen{\"a}herte {D}arstellung der {I}rrationalzahlen durch
  rationale {B}r{\"u}che.
\newblock {\em Math. Ann.}, 39:279--284, 1891.

\bibitem{Jaffard}
S.~Jaffard.
\newblock The spectrum of singularities of {R}iemann's function.
\newblock {\em Rev. Mat. Iberoam.}, 12(2):441--460, 1996.

\bibitem{Jaffard97}
S.~Jaffard.
\newblock Old friends revisited: the multifractal nature of some classical
  functions.
\newblock {\em J. Fourier Anal. Appl.}, 3(1):1--22, January 1997.

\bibitem{Martin}
S.~Jaffard and B.~Martin.
\newblock Multifractal analysis of the brjuno function.
\newblock {\em Invent. Math.}, 212(1):109--132, October 2017.

\bibitem{Leonarduzzi}
R.~Leonarduzzi, H.~Wendt, S.~Jaffard, S.G. Roux, M.E. Torres, and P.~Abry.
\newblock Extending multifractal analysis to negative regularity: P-exponents
  and p-leaders.
\newblock In {\em 2014 IEEE International Conference on Acoustics, Speech and
  Signal Processing (ICASSP)}, pages 305--309. IEEE, May 2014.

\bibitem{Loosveldt21}
L~Loosveldt and S~Nicolay.
\newblock Generalized spaces of pointwise regularity: toward a general
  framework for the {WLM}.
\newblock {\em Nonlinearity}, 34(9):6561--6586, August 2021.

\bibitem{Loosveldt:22}
L.~Loosveldt and S.~Nicolay.
\newblock Some prevalent sets in multifractal analysis: How smooth is almost
  every function in {$T_p^\alpha (x)$}?
\newblock {\em J. Fourier Anal. Appl.}, 28(4), July 2022.

\bibitem{Roth}
K.F. Roth.
\newblock Rational approximations to algebraic numbers.
\newblock {\em Mathematika}, 2(1):1--20, 1955.

\bibitem{Thomae}
J.~Thomae.
\newblock {\em Einleitung in die Theorie der bestimmten Integrale}.
\newblock Louis Nebert, 1875.

\bibitem{Trifonov}
V.~Trifonov, L.~Pasqualucci, R.~Dalla-Favera, and R.~Rabadan.
\newblock Fractal-like distributions over the rational numbers in
  high-throughput biological and clinical data.
\newblock {\em Sci. Rep.}, 1(1), December 2011.

\end{thebibliography}
\bibliographystyle{plain}
\end{document}